\documentclass{amsart}
\usepackage{amsmath,amsthm,amssymb}

\newcommand{\BB}{\mathcal{B}}
\newcommand{\cB}{\mathcal{B}}
\newcommand{\CC}{\mathcal{C}}
\newcommand{\cX}{\mathcal{X}}
\newcommand{\cH}{\mathcal{H}}
\newcommand{\cC}{\mathcal{C}}
\newcommand{\rv}[1]{\mathbf{#1}}

\renewcommand{\:}{\colon}
\newcommand{\e}{\mathrm{e}}
\newcommand{\floor}[1]{\left\lfloor #1 \right\rfloor}

\newcommand{\delete}{\backslash}
\newcommand{\del}{\delete}

\newcommand{\expect}[2][]{\mathbb{E}_{#1}\left[#2\right]}

\newcommand{\var}[2][]{\mathrm{Var}_{#1}\left(#2\right)}

\newcommand{\randbal}{\mathbb{B}}

\newcommand{\bC}{\mathbb{C}}
\newcommand{\bZ}{\mathbb{Z}}
\newcommand{\bF}{\mathbb{F}}

\newcommand{\sym}[1]{\mathcal{S}_{{#1}}}

\newtheorem{theorem}{Theorem}[section]
\newtheorem{lemma}{Lemma}[section]
\newtheorem{claim}{}[lemma]
\newtheorem{corollary}{Corollary}[section]

\newtheorem{conjecture}{Conjecture}

\title{On the number of biased graphs}

\author[Nelson]{Peter Nelson}
\author[van der Pol]{Jorn van der Pol}
\address{University of Waterloo, Waterloo, Canada}

\begin{document}
	
\begin{abstract}
	A \emph{biased graph} is a graph $G$, together with a distinguished subset $\cB$ of its cycles so that no theta-subgraph of $G$ contains precisely two cycles in $\cB$. A large number of biased graphs can be constructed by choosing $G$ to be a complete graph, and $\cB$ to be an arbitrary subset of its Hamilton cycles. We show that, on the logarithmic scale, the total number of simple biased graphs on $n$ vertices does not asymptotically exceed the number that can be constructed in this elementary way.
\end{abstract}	
\maketitle
\section{Introduction}

A \emph{theta-graph} is a graph that is the union of three internally disjoint paths between distinct vertices $u$ and $v$. Such a graph contains exactly three cycles. (i.e. connected $2$-regular subgraphs). 
A \emph{biased graph} is a pair $(G,\cB)$, where $G$ is a finite graph, and $\cB$ is a subset of the collection of cycles of $G$ for which no theta-subgraph of $G$ contains precisely two cycles in $\cB$. We call the cycles in $\cB$ the \emph{balanced} cycles. If $G$ contains no loops or multiple edges, we call $(G,\cB)$ a \emph{simple} biased graph.

 This paper considers the number of simple biased graphs on a fixed $n$-element vertex set. Let $\rv{B}_n$ denote the collection of all simple biased graphs $(G,\cB)$ for which $V(G) = [n] = \{1, \dotsc, n\}$. One can easily see that $\rv{B}_n$ is rather large, even when we restrict to just its members for which $G = K_n$ is a complete graph. For $n \ge 3$, the clique $K_n$ has exactly $\frac{1}{2}(n-1)!$ Hamilton cycles, and no theta-subgraph of $K_n$ contains more than one Hamilton cycle. It follows that $(K_n,\cB)$ is a biased graph whenever $\cB$ contains only Hamilton cycles. There are $2^{\frac{1}{2}(n-1)!}$ such $\cB$, and therefore $|\rv{B}_n| \ge 2^{\frac{1}{2}(n-1)!}$. Our main theorem shows that this lower bound is in fact correct up to lower-order terms in the exponent.\protect\footnote{Throughout this paper, we use $\log$ to denote the binary logarithm.}
\begin{theorem}\label{main}
	$2^{\frac{1}{2}(n-1)!} \le |\rv{B}_n| \le 2^{\frac{1}{2}(n-1)!\left(1 + 21\sqrt{\frac{\log n}{n}}\right)}$ for all $n \ge 3$.
\end{theorem}

We have made no particular effort to optimize the value of the constant $12$, which can likely be reduced at the expense of increasing the lower bound for $n$. 

Biased graphs were defined as an abstraction of the notion of a graph with group labellings on the edges. These are usually defined in terms of a graph with oriented edges; we give an equivalent definition where the vertex set has a canonical ordering. A \emph{group-labelling} of a graph $G = (V,E)$ with $V = [n]$ is a function $\gamma \: E \to \Gamma$ for some group $\Gamma$. Given such a function, for each cycle $C = v_1v_2\dotsc v_kv_1$ of $G$, let $\sigma(C) = \prod_{i \in \bZ_k} \gamma(v_iv_{i+1})^{a_i}$, where $a_i = 1$ if $v_{i+1} > v_i$ and $a_i = -1$ otherwise. We say that $C$ is \emph{balanced} with respect to $\gamma$ if $\sigma(C)$ is the identity in $\Gamma$; this definition ostensibly depends on the choice of ordering of the vertices in $C$, but turns out to be independent of this choice. It is shown in~\cite{Zaslavsky1989} that the collection  $\cB$  of balanced cycles in this sense obeys the theta-property; i.e.\ $(G,\cB)$ is a biased graph. A biased graph $(G,\cB)$ with $V(G) = [n]$ is \emph{group-labellable} if it arises from a group-labelling of $G$ in this way. Using quite different techniques from those in Theorem~\ref{main}, we bound the number of biased graphs arising from labellings over an abelian group. 

\begin{theorem}\label{abelian}
	For each $n \ge 1$, the number of simple biased graphs with vertex set $[n]$ that arise from an abelian-group-labelling is at most $2^{\frac{1}{4}n^5 \log n}$. 
\end{theorem}

In the light of Theorem~\ref{main}, this shows that such graphs are rare indeed. We have not obtained such a theorem for general groups, but we conjecture that group-labellable biased graphs are still rare.

\begin{conjecture}\label{grouplabel}
	If $\rv{H}_n$ is the set of simple group-labellable biased graphs with vertex set $[n]$, then $\lim\limits_{n \to \infty}|\rv{H}_n|/|\rv{B}_n| = 0$.
\end{conjecture}

\section{The overlap graph}

For each integer $n$, let $K_n$ denote the complete graph with vertex set $[n] = \{1,\dotsc, n\}$, and let $\cC$ be the set of cycles of $K_n$. Let $\Omega_n$ be the graph on vertex set $\cC$, in which two vertices $C,C'$ are adjacent if and only if some theta-subgraph of $K_n$ contains $C$ and $C'$ (or, equivalently, the cycles $C$ and $C'$ intersect in a nonempty path). Call this the \emph{overlap graph} of order $n$.

Summing the number of $(n-k)$-cycles for each $k$, we see that the number of vertices in $\Omega_n$ is $\sum_{k=0}^{n-3} \frac{n!}{2k!(n-k)}$. We will need some estimates on this quantity; to this end, for each $n \ge 3$ define $S_n = \sum_{k=0}^{n-3}\frac{1}{k!(n-k)}$
\begin{lemma}\label{bound_sn}
	$\frac{\e}{n} < S_n < \frac{\e}{n} + \frac{5}{n^2}$ for all $n \ge 5$. 
\end{lemma}
\begin{proof}
	The statement is easy to check when $n = 5$; let $n > 5$. We have 
	\[nS_n = \sum_{k=0}^{n-3} \frac{1}{k!} + \sum_{k=0}^{n-3} \frac{k}{k!(n-k)} = \sum_{k=0}^{n-3}\frac{1}{k!} + \sum_{k=0}^{n-4} \frac{1}{k!(n-1-k)} = \sum_{k=0}^{n-3} \frac{1}{k!} + S_{n-1},\]
	which implies $\e - \frac{2}{(n-3)!} + S_{n-1} < nS_n < \e + S_{n-1}$. Using $\frac{2}{(n-3)!} < \frac{\e}{n-1}$ and induction we have
	$nS_n > \e - \frac{2}{(n-3)!} + S_{n-1} > \e - \frac{\e}{n-1} + \frac{\e}{n-1},$
	which gives the lower bound. 
	
	For the upper bound, first note that $n \ge 6$ implies that $ 2 - \frac{5}{n} > \frac{5}{n-1}$, which gives $\frac{5}{n} - \frac{\e}{n-1} > \frac{5}{n} - \frac{3}{n-1} = \frac{1}{n-1}\left(2-\frac{5}{n}\right) > \frac{5}{(n-1)^2}$.  Using this and induction, we get $nS_n < \e + S_{n-1} < \e + \frac{\e}{n-1} + \frac{5}{(n-1)^2} < \e + \frac{5}{n}$. Therefore $S_n < \frac{\e}{n} + \frac{5}{n^2}$ as required. 
\end{proof}

We use the upper bound in the following estimate freely. 
\begin{corollary}
	For all $n \ge 5$ we have $\frac{e}{2}(n-1)! < |V(\Omega_n)| < 2(n-1)!$. 
\end{corollary}
\begin{proof}
	As observed, we have $|V(\Omega_n)| = \frac{n!}{2}S_n$. The claimed bounds follow from Lemma~\ref{bound_sn} and the fact that $\frac{e}{n} + \frac{5}{n^2} < \frac{3}{n} + \frac{5}{n^2} \le \frac{4}{n}$ for $n \ge 5$. 
\end{proof}

We say a set $A$ of vertices \emph{spans} an edge $e$ of a graph if $A$ contains both ends of $e$. Call a subset $\cB \subseteq V(\Omega_n)$ \emph{scarce} if it does not span any edge (i.e.\ it is a stable set of $\Omega_n$); equivalently, every theta-subgraph of $K_n$ contains at most one cycle from~$\cB$. Clearly, if $\cB$ is scarce, then $(K_n, \cB)$ is a biased clique; we call a biased clique $(K_n, \cB)$ \emph{scarce} if and only if $\cB$ is scarce.

The set of Hamilton cycles in $K_n$ is scarce. We show that in fact it is the unique scarce biased clique with the maximum number of balanced cycles. Let $\sym{n}$ denote the set of permutations of $[n]$.

\begin{lemma}\label{lemma:lym}
	If $n \ge 3$ and $(K_n,\cB)$ is a scarce biased clique, then $|\cB| \le \tfrac{1}{2}(n-1)!$. If equality holds, then $\cB$ is the set of Hamilton cycles of $K_n$. 
\end{lemma}

\begin{proof}
	
	For each $3 \le k \le n$, write $\BB_k$ for the set of $k$-cycles in $\BB$, and 
	define functions $\Psi_1 \colon \cB \to 2^{\sym{n}}$ and $\Psi_2\colon (\cB - \cB_n) \to 2^{\sym{n}}$ as
	\begin{align*}
	\Psi_1(C) &= \{\sigma \in \sym{n} \: \text{$(\sigma(1), \ldots, \sigma(|C|))$ is a cyclic ordering of $C$}\}, \\
	\Psi_2(C) &= \{\sigma \in \sym{n} \: \text{$(\sigma(2), \ldots, \sigma(|C|+1))$ is a cyclic ordering of $C$}\}, 
	\end{align*}
	Let $\Psi(C) = \Psi_1(C) \cup \Psi_2(C)$ for $C \in \cB - \cB_n$ and $\Psi(C) = \Psi_1(C)$ for $C \in \cB_n$.

	We claim that the $\Psi$-images of distinct cycles in $\BB$ are disjoint. If this is not the case, then there are cycles $C_1$ and $C_2$, integers $1 \le i \le j \le 2$, and a permutation $\sigma$ such that $\sigma \in \Psi_i(C_1) \cap \Psi_j(C_2)$. We consider three different cases, depending on $(i,j)$.
	
	If $(i,j) = (1,1)$, then $|C_1| \neq |C_2|$ (otherwise $C_1 = C_2$). Let $\ell = \min\{|C_1|, |C_2|\}$. The cycles $C_1$ and $C_2$ intersect in the path $\sigma(1)\sigma(2)\cdots\sigma(\ell)$, and hence are adjacent in $\Omega_n$: a contradiction.
	
	If $(i,j) = (2,2)$, a contradiction follows by a similar argument.
	
	If $(i,j) = (1,2)$, let $\ell = \min\{|C_1|, |C_2| + 1\}$. The cycles $C_1$ and $C_2$ intersect in the path $\sigma(2)\sigma(3)\cdots\sigma(\ell)$, hence are adjacent in $\Omega_n$: again a contradiction.
	
	Thus, $\Psi$ encodes each cycle as a collection of permutations, and these collections are pairwise disjoint, it follows that $\sum_{C \in \BB} |\Psi(C)| \le n!$. As $|\Psi(C)| = 4|C|(n-|C|)!$ if $|C| < n$ and $|\Psi(C)| = 2n$ if $|C| = n$, this yields	
	\[n! \ge \sum_{C \in \cB}|\Psi(C)| = \sum_{k = 3}^{n-1} 4k(n-k)!|\cB_k| + 2n|\cB_n| \]

	Now $4k(n-k)! \ge 4k(n-k) \ge 4(n-1) > 2n$ for all $3 \le k < n$; it follows that $n! \ge 2n\sum_{k=3}^n|\cB_k| = 2n|\cB|$, so $|\cB| \le \tfrac{1}{2}(n-1)!$ as required. If equality holds, then clearly $\cB = \cB_n$. 
\end{proof}

We now prove that every subset of $V(\Omega_n)$ that is significantly larger than the largest stable set necessarily spans a substantial number of edges. 

\begin{lemma}\label{supersat}
	Let $n \ge 12$ be an integer and let $\alpha \ge \frac{88}{n}$. If $\BB \subseteq V(\Omega_n)$ satisfies $|\BB| \ge (1+\alpha)\frac{1}{2}(n-1)!$, then $\BB$ spans at least $\tfrac{\alpha}{8} n!$ edges in $\Omega_n$.
\end{lemma}

\begin{proof}
 Let  $\Omega = \Omega_n$, and suppose for a contradiction that some set $\BB \subseteq V(\Omega)$ satisfies $|\BB| \ge (1+\alpha)\frac{1}{2}(n-1)!$, but $\BB$ spans fewer than $ \tfrac{1}{8}\alpha n!$ edges in $\Omega$.   Note that, since $|V(\Omega)| < 2(n-1)!$, we have $\alpha < 3$. 	
Let $\cC = V(\Omega)$. Define functions $\Phi, \Phi_1, \Phi_2\colon \CC \to 2^{\sym{n}}$ as
	\begin{align*}
	\Phi_1(C) &= \{\sigma \in \sym{n} \: \text{$(\sigma(1), \ldots, \sigma(|C|))$ is a cyclic ordering of $C$}\}, \\
	\Phi_2(C) &= \{\sigma \in \sym{n} \: \text{$(\sigma(n-|C|+1), \ldots, \sigma(n))$ is a cyclic ordering of $C$}\}, 
	\end{align*}
	and $\Phi(C) = \Phi_1(C) \cup \Phi_2(C)$. Note that the sets $\Phi_1(C)$ and $\Phi_2(C)$ are disjoint if $|C| < n$ and equal if $|C| = n$; it follows that $|\Phi(C)| = 4|C|(n-|C|)!$ if $3 \le |C| < n$, and $|\Phi(C)| = 2n$ if $|C| = n$. Note also that for each $k < n$ and $\sigma \in \sym{n}$, there are exactly two $k$-cycles $C$ for which $\sigma \in \Phi(C)$, and there is exactly one $n$-cycle $C$ with $\sigma \in \Phi(C)$. Thus $|\Phi(\cC')| = 2n|\cC'|$ for each set $\cC'$ of $n$-cycles. Finally, observe that $|\Phi(C) \cap \Phi(C')| \le 2$ for all distinct $C,C' \in \cC$ for which $|C| + |C'| > n$.

	For each $3 \le k \le n$, let $\cB_k$ be the set of $k$-cycles in $\cB$, 
	and for each $i \in \{0,1,2\}$, let $P_{k,i}$ be the set of all $\sigma \in \sym{n}$ for which $|\{C \in \BB_k \colon \sigma \in \Phi(C)\}| = i$. Since each $\sigma \in \sym{n}$ is in $\Phi_1(C)$ for at most one $C \in \cB_k$ and is in $\Phi_2(C)$ for at most one $C \in \cB_k$, the sets $P_{k,0}, P_{k,1},P_{k,2}$ partition $\sym{n}$. 

	\begin{claim}\label{claim:P2-small}
		$|P_{k,2}| \le \frac{\alpha}{2} n!$ for each $k \ge \tfrac{n+2}{2}$. 
	\end{claim}
	
	\begin{proof}[Proof of claim]
		Let $\sigma \in P_{k,2}$, and let $C,C' \in \cB_k$ be the distinct cycles for which $\sigma \in \Phi(C) \cap \Phi(C')$. As $C$ and $C'$ intersect in at least $2k-n \ge 2$ consecutive elements, the cycles $C$ and $C'$ intersect in a path of at least two vertices, so are adjacent in $\Omega$. Moreover, we have $|C| + |C'| = 2k > n$, so $|\Phi(C) \cap \Phi(C')| \le 4$. It follows that $\BB_k$ spans at least $\tfrac{1}{4}|P_{k,2}|$ edges. Thus $\tfrac{1}{4}|P_{k,2}| \le \frac{\alpha}{8}n!$, giving the claim. 
	\end{proof}

	\begin{claim}\label{claim:Phi-intersection-small}
		$|\Phi(\BB_k) \cap \Phi(\BB_n)| \le \tfrac{\alpha}{2} n!$ for all $3 \le k < n$.
	\end{claim}
	
	\begin{proof}[Proof of claim]
		Let $C \in \BB_k$, $C' \in \BB_n$, and let $\sigma \in \Phi(C) \cap \Phi(C')$. Since $\sigma(1), \dotsc, \sigma(n)$ is a cyclic ordering of $C'$ and either $\sigma(1), \dotsc, \sigma(|C|)$ or $\sigma(n-|C|+1), \dotsc, \sigma(n)$ is a cyclic ordering of $C$, the cycles $C,C'$ are adjacent in $\Omega$. Since $|C| + |C'| > n$, we have $|\Phi(C) \cap \Phi(C')| \le 4$; it follows that $\Omega$ contains at least $\tfrac{1}{4}|\Phi(\BB_k) \cap \Phi(\BB_n)|$ edges of this form. Thus $|\Phi(\cB_k) \cap \Phi(\cB_n)| \le \tfrac{\alpha}{2}n!$, as required. 
	\end{proof}
	
	Recall that the number of $k$-cycles in $K_n$ is $\tfrac{n!}{2k(n-k)!}$. For $3 \le k \le n$, let $\beta_k = |\BB_k|\frac{2k(n-k)!}{n!}$, so $0 \le \beta_k \le 1$. Note that $|\Phi(\cB_n)| = 2n|\cB_n| = \beta_n n!$. 
	
	\begin{claim}\label{claim:beta-k-bound}
		$\beta_k \le \frac{1}{2}(1-\beta_n + \alpha)$ for all $\tfrac{n+2}{2} \le k < n$. 
	\end{claim}
	
	\begin{proof}[Proof of claim]
		Using Claim~\ref{claim:Phi-intersection-small}, we have
		\[|\Phi(\BB_k)| = |\Phi(\BB_k) \cup \Phi(\BB_n)| + |\Phi(\BB_k) \cap \Phi(\BB_n)| - |\Phi(\BB_n)| \le n! + \tfrac{\alpha}{2} n! - \beta_n n!,\]
		so $|\Phi(\cB_k)| \le (1+\frac{\alpha}{2}-\beta_n)n!$.  Note that 
		\[|P_{k,1}| + 2|P_{k,2}| = \sum_{i=0}^2 i |P_{k,i}| = \sum_{C \in \cB_k} |\Phi(C)| = 4k(n-k)!|\BB_k|\] and that $|P_{k,1}| + |P_{k,2}| = |\Phi(\BB_k)|.$ Using \ref{claim:P2-small}, we have
		\begin{equation*}
			\begin{split}
				2\beta_k n!
					&= 4k(n-k)!|\BB_k| \\
					&= |P_{k,1}| + 2|P_{k,2}| \\
					&= |\Phi(\BB_k)| + |P_{k,2}| \\
					&\le (1+\tfrac{\alpha}{2}-\beta_n)n! + \tfrac{\alpha}{2}n! \\
					&= (1+\alpha-\beta_n)n!.
			\end{split}
		\end{equation*}
		The claim follows immediately upon dividing both sides by $2n!$.
	\end{proof}
	
	It follows from~\ref{claim:beta-k-bound} and the hypothesis that
	\begin{align*}
	\frac{1+\alpha}{2}(n-1)! &\le |\cB| = \sum_{0 \le k \le n-3}|\cB_{n-k}| \\
	&= \frac{n!}{2}\sum_{0 \le k \le n-3}\frac{\beta_{n-k}}{k!(n-k)}\\
	&\le \frac{n!}{2}\left(\sum_{k = \floor{n/2}}^{n-3} \frac{1}{k!(n-k)} + \sum_{k = 1}^{\floor{n/2}-1}\frac{\beta_{n-k}}{k!(n-k)} + \frac{\beta_n}{n}\right)\\
	&\le \frac{n!}{2}\left(\frac{1}{3}\sum_{k \ge \floor{n/2}}\frac{1}{k!} + \frac{1-\beta_n + \alpha}{2}\sum_{k=1}^{\floor{n/2}-1}\frac{1}{k!(n-k)} + \frac{\beta_n}{n}\right)\\
	&\le \frac{n!}{2}\left(\frac{2}{3\floor{n/2}!} + \frac{1-\beta_n + \alpha}{2}\left(S_n-\frac{1}{n}\right) + \frac{\beta_n}{n}\right)\\
	&\le \frac{n!}{2}\left(\frac{1}{n^2}+ \frac{1-\beta_n+\alpha}{2} \left(\frac{\e-1}{n} + \frac{5}{n^2}\right) + \frac{\beta_n}{n}\right),
	\end{align*}
	where the last line uses $\floor{n/2}! > n^2$ (for $n \ge 12$) and Lemma~\ref{bound_sn}. Since $0 \le \beta_n \le 1$ and $\alpha < 3$, this gives
	\[1+ \alpha \le \tfrac{1}{2}(1+\alpha)(\e-1)  + \beta_n\left(1 - \tfrac{\e-1}{2}\right) + \tfrac{11}{n},\]
	and so $\frac{11}{n} \ge \frac{3-\e}{2}(1 + \alpha - \beta_n) \ge \tfrac{3-\e}{2}\alpha$. Using $3-\e > \frac{1}{4}$, we obtain a contradiction to the hypothesis that $\alpha \ge \frac{88}{n}$.
\end{proof}

\section{Scarce biased cliques}

Let $\mathbf{S}_n$ denote the collection of scarce biased cliques on $n$ vertices. In this section we prove the following. 
\begin{theorem}\label{thm:enum-scarce}
	$|\mathbf{S}_n| \le 2^{\frac{1}{2}(n-1)!\left(1 + 19 \sqrt{\frac{\log n}{n}}\right)}$ for all $n \ge 16$. 
\end{theorem}

Our proof is a standard application of the container method; see~\cite{Samotij2015} for an introduction to and background of this technique. The main tool is the following lemma, which essentially allows us to find a concise description of every scarce biased clique $\cB$, in terms of a set $\psi(\cB)$ of size $o((n-1)!)$, and a subset of a set $\phi(\psi(\cB))$ of size $(\frac{1}{2}+o(1))(n-1)!$.

\begin{lemma}\label{lemma:container-tech}
	Let $n \ge 16$ be an integer. Let $\cC = V(\Omega_n)$ and 
	\[s = \frac{4(n-1)!}{\sqrt{n \log n}} \text{\ \ \ and\ \ \ } a = \left(1 + 11\sqrt{\frac{\log n}{n}}\right)\tfrac{1}{2}(n-1)!.\]
	There exist functions $\psi\colon 2^{\CC}\to\binom{\CC}{\le s}$  and $\phi\colon2^{\cC}\to\binom{\CC}{\le a}$ such that each scarce biased clique $(K_n, \BB)$ satisfies $ \BB \subseteq \phi(\psi(\BB))$.
\end{lemma}

\begin{proof} 
	
	Let $n \ge 16$ and $\alpha = 11\sqrt{\frac{\log n}{n}}$, noting that $\alpha \ge \tfrac{88}{n}$ . Let $\Omega = \Omega_n$, and fix a linear order $\sqsubseteq$ on $\cC = V(\Omega)$. For each set $A \subseteq \CC$, write $C_A^*$ for the vertex of maximum degree in the induced subgraph $\Omega[A]$, where ties are broken using $\sqsubseteq$. Define a function $f\colon 2^{\CC} \times 2^{\CC} \times 2^{\CC} \to 2^{\CC} \times 2^{\CC}$ by

	\begin{equation*}
	f(S,A,K) =
	\begin{cases}
	(S,A) & \text{if $|A| \le a$} \\
	(S\cup\{C_A^*\}, A\setminus N_\Omega(C_A^*)) & \text{if $|A| > a$ and $C_A^* \in K$} \\
	(S, A\setminus\{C_A^*\}) & \text{otherwise}.
	\end{cases}
	\end{equation*}

	For each $\cB \subseteq \cC$, recursively define sequences $S_i = S_i(\cB)$ and $A_i = A_i(\cB)$ by $S_0 = \varnothing$, $A_0 = \CC$, and
	\begin{equation*}
		(S_{i+1}, A_{i+1}) = f(S_i, A_i, \cB).
	\end{equation*}
	Since $A_{i+1} \subseteq A_i$, there exists $i_0$ such that $(S_i, A_i) = (S_{i_0}, A_{i_0})$ for all $i \ge i_0$; define the functions $\psi$ and $\phi$ by $\psi(\cB) = S_{i_0}$ and $\phi(\cB) = A_{i_0}$. 
	
	\begin{claim}
		$|\psi(\cB)| \le a$ and $|\phi(\cB)| \le s$ for all $\cB \subseteq \cC$. 
	\end{claim}
	\begin{proof}[Proof of claim]
		That $|\psi(\cB)| \le a$ follows immediately from the construction. For the second part, note that $|S_{i_0}|$ is equal to the number of $i$ for which $S_{i+1} \ne S_i$. For each such $i$, we have $|A_i| > a$ and $C^*_{A_i} \in \BB$, so Lemma~\ref{supersat} implies that $A_i$ spans at least $\frac{\alpha }{8}n!$ edges of $\Omega$. Therefore $C^*_{A_i}$ has at least $\frac{\alpha n!}{4|A_i|} \ge \frac{\alpha n!}{4|\cC|} \ge \frac{1}{8}n\alpha$ neighbours in $A_i$, and thus $|A_{i+1}| \le |A_i| - \frac{1}{8}n\alpha$. This occurs for each of the $|S_{i_0}|$ distinct values of $i$ for which $S_{i+1} \ne S_i$. Since the sequence $\left(|A_i| \:i \ge 0\right)$ is weakly decreasing, it follows that 
		\[0 \le |A_{i_0}| \le |A_0| - |S_{i_0}|\cdot \tfrac{1}{8}n\alpha \le 2(n-1)! - |S_{i_0}| \tfrac{1}{8}n\alpha,\]
		and so $|\phi(\cB)| = |S_{i_0}| \le \frac{16(n-1)!}{n\alpha} \le s$.
	\end{proof}
	
	If $\cB$ is a scarce biased clique, then, since $\cB$ is a stable set of $\Omega_n$, for each $i$ the elements of $A_i \del A_{i+1}$ are all nonelements of $\cB$; it follows by an inductive argument that $\cB \subseteq A_{i_0} = \phi(\cB)$. 
	
	Now consider the sequences $S_i(\psi(\cB))$ and $A_i(\psi(\cB))$. Another inductive argument shows that they coincide with the sequences $S_i(\cB)$ and $A_i(\cB)$ respectively, and thus $\phi(\psi(\cB)) = \phi(\cB)$, and so $\cB \subseteq \phi(\psi(\cB))$ as required. 
\end{proof}

We are now ready to prove Theorem~\ref{thm:enum-scarce}.

\begin{proof}[Proof of Theorem~\ref{thm:enum-scarce}]
	Obtain $s$ and $a$, as well as functions $\psi\colon 2^{\cC} \to \binom{\cC}{\le a}$ and $\phi\colon 2^{\cC} \to \binom{\cC}{\le s}$ as in Lemma~\ref{lemma:container-tech}. The number of scarce biased cliques $\cB \subseteq V(\Omega_n)$ is at most
	\begin{equation*}
		\left|\left\{(S,A) : S \in \binom{\CC}{\le s}, A \subseteq \phi(S)\right\}\right| \le \binom{|\CC|}{\le s} 2^a.
	\end{equation*}
	
	Using $|\cC| \le 2(n-1)!$ and $n \ge 8$, we have 
	\[\frac{\e |\cC|}{s} \le \frac{6(n-1)!}{s} = \frac{3\sqrt{n \log n}}{2} \le 2^{\log n}.\] Now, using a standard bound on sums of binomial coefficients, we have 
\[
		\binom{|\CC|}{\le s} \le \left(\frac{\e |\CC|}{s}\right)^{s} \le \left(2^{\log n}\right)^{\frac{4(n-1)!}{\sqrt{n \log n}}} =  2^{8\sqrt{\frac{\log n}{n}} \frac{1}{2}(n-1)!}.
\]
	The theorem follows as $a = \left(1+11\sqrt{\frac{\log n}{n}}\right)\frac{1}{2}(n-1)!$.
\end{proof}

\section{Biased graphs}

Let $\mathbf{K}_n$ denote the collection of biased cliques on $n$ vertices. 

\begin{theorem}
	$|\mathbf{K}_n| \le |\mathbf{S}_n| \cdot 2^{(n-1)! \frac{n^2}{6\floor{n/3}!}}$ for all $n$. 
\end{theorem}
\begin{proof}
	Let $\cC$ be the set of cycles of $K_n$ and let $\prec$ be a linear ordering of $\cC$ that refines the partial ordering by length (i.e. $C \prec C'$ whenever $|C| < |C'|$). Note that if $C_1,C_2,C_3$ are the cycles in a $\Theta$-subgraph $H$ of $K_n$ whose degree-$3$ vertices are $u$ and $v$, then each vertex in $H - \{u,v\}$ is in exactly two of the $C_i$, so $\sum_{i=1}^3 |C_i| \le 2|V(H)| + 2 \le 2n+2$. It follows that one of the $C_i$ has length at most $\frac{2}{3}(n+1)$. Let $\cC'$ be the set of cycles in $\cC$ of length at most $\frac{2}{3}(n+1)$ and let $r = |\cC'|$; note that 
	\[r = \sum_{k=3}^{\floor{2(n+1)/3}} \frac{n!}{2k(n-k)!} \le (n-1)!\sum_{k=3}^{\floor{2(n+1)/3}} \frac{n}{6(n-k)!} \le (n-1)! \frac{n^2}{6\floor{n/3}!}.\]

	Let $\cH$ be the collection of all triples $(C_1,C_2,C_3)$ for which $C_1 \prec C_2 \prec C_3$, and $C_1,C_2,C_3$ are the cycles of a $\Theta$-subgraph of $K_n$. By the above, we have $C_1 \in \cC'$. For each biased clique $\cB \subseteq \cC$, let $\psi(\cB)$ be obtained from $\cB$ by removing $C_1$ and $C_3$ for each triple $(C_1,C_2,C_3) \in \cH$ for which $\{C_1,C_2,C_3\} \subseteq \cB$. Since $|\{C_1,C_2,C_3\} \cap \cB| \in \{0,1,3\}$ for all $(C_1,C_2,C_3) \in \cH$, it follows that $|\{C_1,C_2,C_3\} \cap \psi(\cB)| \le 1$ for each $(C_1,C_2,C_3)$, so $\psi(\cB)$ is a scarce biased clique. 
	
	We now show that for each scarce biased clique $\cB'$ and each set $\cX \subseteq \cC'$, there is at most one biased clique $\cB$ for which $\psi(\cB) = \cB'$ and $\cB \cap \cC' = \cX$. It will follow that $|\mathbf{K}_n| \le |\mathbf{S}_n| 2^{|\cC'|} = |\mathbf{S}_n|2^r$, which gives the theorem by our estimate on $r$. 
	
	Suppose that the claimed statement fails, so there is a scarce biased clique $\cB'$ and a pair of distinct biased cliques $\cB_1,\cB_2$ for which $\psi(\cB_1) = \psi(\cB_2) = \cB'$ while $\cB_1 \cap \cC' = \cB_2 \cap \cC'$. Let $\cC''$ be a maximal initial segment of $\cC$ with respect to $\prec$ for which $\cB_1 \cap \cC'' = \cB_2 \cap \cC''$; we have $\cC' \subseteq \cC'' \ne \cC$ by assumption. Let $C \in \cC - \cC''$ be minimal with respect to $\prec$. The maximality in the choice of $\cC''$ implies that $C$ belongs to exactly one of $\cB_1$ and $\cB_2$; say $C \in \cB_1 - \cB_2$. 
	
	Since $\cB = \psi(\cB_2) \subseteq \cB_2$ and $C \notin \cB_2$, we have $C \notin \cB = \psi(\cB_1)$. Since $C \in \cB_1$, there
	is some $(C_1,C_2,C_3) \in \cH$ for which $\{C_1,C_2,C_3\} \subseteq \cB_1$ while $C \in \{C_1,C_3\}$. Since $C_1 \in \cC'$ and $C \notin \cC'$, this gives $C = C_3$. Now $C_1 \prec C_2 \prec C$, giving $\{C_1,C_2\} \subseteq \cC''$, and so $\{C_1,C_2\} \cap \cB_2 = \{C_1,C_2\} \cap \cB_1 = \{C_1,C_2\}$. Therefore $C_1,C_2 \in \cB_2$ and thus $C = C_3 \in \cB_2$ by the theta-property, a contradiction. 
\end{proof}

Since $\frac{n^2}{6\floor{n/3}!} \le \frac{1}{2}\sqrt{\frac{\log n}{n}}$ for $n \ge 18$, combining the above result with Theorem~\ref{thm:enum-scarce} gives the following. 

\begin{theorem}\label{countcliques}
	$|\mathbf{K}_n| \le 2^{\frac{1}{2}(n-1)!\left(1 + 20\sqrt{\frac{\log n}{n}}\right)}$ for all $n \ge 18$. 
\end{theorem}

Finally, we prove Theorem~\ref{main}.
\begin{proof}[Proof of Theorem~\ref{main}]
	There are exactly $9 < 2^{1 + 21\sqrt{\frac{\log 3}{3}}}$ simple biased graphs on three vertices, so we may assume that $n > 3$. If $(G,\cB)$ is a simple biased graph with $V(G) = [n]$, then $(K_n,\cB)$ is a biased clique on $n$ vertices, since for any theta-subgraph $H$ of $K_n$ containing precisely two cycles in $\cB$, we have $E(H) \subseteq \cup_{C \in \cB} E(C) \subseteq E(G)$, so no such $H$ exists. Since $E(K_n)$ has $2^{\binom{n}{2}}$ subsets, it follows that $|\rv{B}_n| \le 2^{\binom{n}{2}}|\rv{K}_n|$. 
	
	If $n \ge 18$, then using $\binom{n}{2} < \tfrac{1}{2} (n-1)!\sqrt{\frac{\log n}{n}}$, the required bound follows from Theorem~\ref{countcliques}. If $4 \le n \le 17$, then $\frac{\log(n)}{n} \ge \frac{9}{49}$, so $\tfrac{1}{2}\left(1 + 21 \sqrt{\frac{\log n}{n}}\right) \ge 5 > 3$. Using $(n-1)! \ge \binom{n}{2}$, we now have
	\[2^{\frac{1}{2}(n-1)!\left(1 + 21\sqrt{\frac{\log n}{n}}\right)} \ge 2^{3(n-1)!} \ge 2^{2(n-1)! + \binom{n}{2}} \ge 2^{|V(\Omega_n)|} \cdot 2^{\binom{n}{2}} \ge 2^{\binom{n}{2}}|\rv{K}_n|,\]
	giving the bound. 
\end{proof}

\section{Group-labellable biased graphs}

In this section we prove the following stronger version of Theorem~\ref{abelian}:
\begin{theorem}\label{abelian:precise}
	For each $n \ge 3$, the number of simple biased graphs with vertex set $[n]$ that arise from an abelian-group-labelling is at most $(2n! + 1)^{\binom{n}{2}^2}$.
\end{theorem}

Theorem~\ref{abelian:precise} implies Theorem~\ref{abelian}, since $2n!+1 \le n^n$ for all $n \ge 3$.


Let $I$ be a finite set, and let $\mathbf{f} = (f_i : i \in I)$ be a tuple of polynomials in the variables $X_1, X_2, \ldots, X_N$ with coefficients in a field $\bF$. A set $S \subseteq I$ is a \emph{zero-pattern} arising from $\mathbf{f}$ if there exists $w \in \bF^N$ such that $S = \{i \in I : f_i(w) \neq 0\}$. In this case, we call $w$ a \emph{witness} for $S$.

Write $Z_{\bF}(\mathbf{f})$ for the set of zero-patterns arising from $\mathbf{f}$. The following result, due to R\'{o}nyai, Babai, and Ganapathy~\cite{RonyaiBabaiGanapathy2001}, bounds the cardinality of $Z_{\bF}(\mathbf{f})$.

\begin{theorem}\label{zeropatterns}
	Let $\mathbf{f}$ be an $M$-tuple of polynomials in the variables $X_1, X_2, \ldots, X_N$ over a field $\bF$. If each of the polynomials has degree at most $D$, then $\left|Z_{\bF}(\mathbf{f})\right| \le \binom{MD}{N}$.
\end{theorem}

Define an $\binom{n}{2}$-tuple of variables $X = (X_e : e \in E(K_n))$. For a cycle $C = v_1v_2\ldots v_kv_1$ of $K_n$, let $f_C \in \bC[X]$ be the polynomial
\begin{equation*}
	f_C(X) = \prod_{i \in \bZ_k, v_i > v_{i+1}} X_{v_iv_{i+1}} - \prod_{i \in \bZ_k, v_i < v_{i+1}} X_{v_iv_{i+1}}.
\end{equation*}

For each simple graph $G$ with vertex set $[n]$, let $\mathbf{f}_G$ denote the tuple comprising $f_C$ for each cycle $C$ of $G$.
Note that $|\mathbf{f}_G| < 2(n-1)!$, as $G$ has at most $|V(\Omega_n)|$ cycles.
The following observation links group-labellable biased graphs with zero-patterns.

\begin{lemma}\label{abelian-zeropattern}
	Let $G = (V,E)$ be a simple graph with $V = [n]$ and $|E| = m$, and let $\cC$ be the set of cycles of $G$. If $(G,\cB)$ is a biased graph that arises from an abelian-group-labelling, then there exist $P_1, \dotsc, P_m \in Z_{\bC}(\mathbf{f}_G)$ such that $\cC - \cB = \bigcup P_i$.
\end{lemma}
\begin{proof}
	Let $\gamma\colon E \to \Gamma$ be an abelian-group-labelling of $G$ that gives rise to $\cB$. Since $|E| = m$, we can restrict $\Gamma$ to the subgroup $\Gamma'$ generated by the image $\gamma(E)$, which, by the fundamental theorem of finitely generated abelian groups, has the form $\Gamma' \cong \bZ^{s} \oplus \bZ_{q_1} \oplus \dotsc \oplus \bZ_{q_t}$, where $s + t \le m$, while $q_1, \dotsc, q_t \ge 1$ are integers. By including some trivial groups, we may assume that $s + t = m$.
	
	Since each cyclic group is a subgroup of $\bC^{\times}$, the group $\Gamma'$ is a subgroup of $(\bC^{\times})^{m}$. For each $i \in [m]$, let $P_i = \{C : f_C(\pi_i \circ \gamma) \ne 0\}$, where $\pi_i$ is the projection map onto the $i$-th co-ordinate. Note that each $P_i$ is a zero-pattern with respect to $\mathbf{f}_G$, witnessed by $\pi_i \circ \gamma$. By construction, we have $\cC - \cB = \bigcup P_i$, as required.  
\end{proof}

\begin{lemma}\label{abelian-countbyedges}
	For all integers $m,n$ with $n \ge 3$ and $0 \le m \le \binom{n}{2}$, the number of abelian-group-labellable simple biased graphs with vertex set $[n]$ and $m$ edges is at most $\binom{{\scriptsize\binom{n}{2}}}{m} \binom{2n!}{m}^m$.
\end{lemma}
\begin{proof}
	Since there are exactly $\binom{{\scriptsize\binom{n}{2}}}{m}$ simple graphs on $[n]$ with $m$ edges, it suffices to show that for each such graph $G$, there are at most $\binom{2n!}{m}^m$ abelian-group-labellable biased graphs on $G$. By Lemma~\ref{abelian-zeropattern}, for every such biased graph $(G, \cB)$, there are zero-patterns $P_1, \ldots, P_m \in Z_{\bC}(\mathbf{f}_G)$ such that $\cC-\cB = \bigcup P_i$, so the number of abelian-group-labellable biased graphs on $G$ is at most $\left|Z_{\bC}(\mathbf{f}_G)\right|^m$. By Theorem~\ref{zeropatterns} (applied here with $\bF=\bC$, $M=|\mathbf{f}_G|$, $N=m$, and $D = n$) and using that $|\mathbf{f}_G| < 2(n-1)!$, the number of zero-patterns arising from $\mathbf{f}_G$ is at most $\binom{2n!}{m}$, which concludes the proof.
\end{proof}

We are now ready to prove Theorem~\ref{abelian:precise}.

\begin{proof}[Proof of Theorem~\ref{abelian:precise}]
	Using Lemma~\ref{abelian-countbyedges} to bound the number of abelian-group-labellable simple biased graphs with a given number $m$ of edges, we obtain the required bound by summing over all possible values of $m$.
	Using that $\binom{2n!}{m} \le (2n!)^{\binom{n}{2}}$ for all $n \ge 3$ and $0 \le m \le \binom{n}{2}$, we find that the number of abelian-group-labellable simple biased graphs with vertex set $[n]$ is at most
	\begin{equation*}
		\sum_{m=0}^{\binom{n}{2}} \binom{\binom{n}{2}}{m} \binom{2n!}{m}^m
			\le \sum_{m=0}^{\binom{n}{2}} \binom{\binom{n}{2}}{m} (2n!)^{\binom{n}{2}m}
			\le (1+2n!)^{\binom{n}{2}^2}. \qedhere
	\end{equation*}
\end{proof}

While it seems difficult to construct a good model for a uniformly random biased clique on $n$ vertices, Theorem~\ref{main} suggests that simply taking a uniformly random subset of the Hamilton cycles as the set of balanced cycles may be a reasonable `approximation'. The remainder of this section shows that, under this significantly simplified model, almost all biased cliques are not labellable over any group, abelian or non-abelian. This can be interpreted as weak evidence for Conjecture~\ref{grouplabel}. 

For $n \ge 8$, call an $n$-vertex graph $R$ a \emph{diamond ring} if it has a Hamilton cycle $v_i: i \in \bZ_n$, and is constructed by choosing vertices $v_i,v_j$ with $j -i \notin \{-2,-1,0,1,2\}$, then adding the four edges $v_iv_{i+2},v_{i+1}v_{i+3},v_jv_{j+2}$ and $v_{j+1}v_{j+3}$. A diamond ring $R$ contains exactly four Hamilton cycles; write $H(R)$ for the set of Hamilton cycles in $R$. For each $C \in H(R)$ there will be a unique choice of the unordered pair $\{i,j\}$ that determines $R$ as described above; it follows that the number of diamond ring subgraphs of $K_n$ is $\tfrac{1}{4}\left(\tfrac{1}{2}(n-1)!\right)\left(\tfrac{1}{2}n(n-5)\right) = \tfrac{1}{16}n!(n-5)$, since there are $\tfrac{1}{2}(n-1)!$ Hamilton cycles and $\tfrac{1}{2}n(n-5)$ choices for $\{i,j\}$. It is easy to verify that if $R$ is a subgraph of a group-labellable biased graph, and three of the cycles in $H(R)$ are balanced, then so is the fourth.

\begin{lemma}\label{lemma:bad-diamonds}
	Let $(G,\BB)$ be a biased graph on $n$ vertices, and let $R$ be an $n$-vertex diamond ring in $G$. If $(G,\cB)$ is group-labellable, then $|\BB \cap H(R)| \neq 3$. 
\end{lemma}

Consider a random biased clique $(K_n, \randbal)$ obtained by including each Hamilton cycle of $K_n$ in $\randbal$ with probability $1/2$, independently of all other Hamilton cycles.

\begin{theorem}
	With high probability (i.e.\ with probability tending to 1 as $n$ tends to infinity) the biased clique $(K_n, \randbal)$ is not group-labellable.
\end{theorem}

\begin{proof}
	Call an $n$-vertex diamond ring subgraph $R$ of $K_n$ bad if $|\randbal \cap H(R)| = 3$, and write $X$ for the (random) number of bad diamond ring subgraphs in $(K_n, \randbal)$. We show, using a straightforward application of the second moment method, that ${X > 0}$ with high probability; the theorem then follows immediately from Lemma~\ref{lemma:bad-diamonds}.
	
	Let $X_R = 1$ if $R$ is bad and $X_R = 0$ otherwise; clearly $X = \sum X_R$, where the sum is over all $n$-vertex diamond ring subgraphs of $K_n$. The clique $K_n$ contains $\frac{1}{16}n!(n-5)$ diamond rings. Each of these is bad with probability $\binom{4}{3}2^{-4}$, so $\expect{X} = \frac{1}{64}n!(n-5) \sim \frac{1}{64}(n+1)!$.
	
	Let $R$ and $R'$ be two diamond ring subgraphs of $K_n$. $X_R$ and $X_{R'}$ are dependent if and only if $R$ and $R'$ share at least one Hamilton cycle. The number of dependent pairs is therefore at most $\frac{1}{2}(n-1)!n^4$. It follows that
	\begin{equation*}
		\var{X}
			\le \expect{X} + \frac{1}{2}(n-1)! n^4
			= o\left(\left(\expect{X}\right)^2\right),
	\end{equation*}
	and hence that $X > 0$ with high probability.
\end{proof}

\bibliographystyle{alpha}
\bibliography{biasedcliques-2020}

\end{document}